\documentclass{IEEEtran}

\usepackage[noadjust]{cite}
\usepackage{graphicx,epsfig}
\usepackage{tikz}
\usetikzlibrary{arrows,positioning} 
\tikzset{
    >=stealth',
    block/.style={
           rectangle,
           rounded corners,
           draw=black, very thick,
           text width=10em,
           minimum height=5em,
           text centered},
    link/.style={
           ->,
           thick,
           shorten <=2pt,
           shorten >=2pt,}
}
\usepackage{amssymb,amsmath,amsfonts,amsthm}
\usepackage[inline]{enumitem}

\theoremstyle{plain}
\newtheorem{thm}{Theorem}
\newtheorem{lem}{Lemma}

\theoremstyle{definition}
\newtheorem{defn}{Definition}

\theoremstyle{remark}
\newtheorem{rem}{Remark}

\newcommand{\R}{{\mathbb R}}
\newcommand{\N}{{\mathbb N}}
\newcommand{\Q}{{\mathcal Q}}

\renewcommand{\S}{{\mathcal S}}
\newcommand{\co}[1]{\overline{co}\left\{#1\right\}}

\DeclareMathOperator*{\argmin}{arg\,min}

\begin{document}
\title{On Switching Stabilizability for Continuous-Time Switched Linear Systems}
\author{Yueyun Lu and Wei Zhang
\thanks{Y. Lu and W. Zhang are with the Department of Electrical and Computer Engineering, The Ohio State University, Columbus, OH 43210. E-mail:\{lu.692,zhang.491\}@osu.edu}}

\maketitle

\begin{abstract}
This paper studies switching stabilization problems for continuous-time switched linear systems. We consider four types of switching stabilizability defined under different assumptions on the switching control input. The most general switching stabilizability is defined as the existence of a measurable switching signal under which the resulting time-varying system is asymptotically stable. Discrete switching stabilizability is defined similarly but requires the switching signal to be piecewise constant on intervals of uniform length. In addition, we define feedback stabilizability in Filippov sense (resp. sample-and-hold sense) as the existence of a feedback law under which closed-loop Filippov solution (resp. sample-and-hold solution) is asymptotically stable. It is proved that the four switching stabilizability definitions are equivalent and their {\em sufficient and necessary} condition is the existence of a piecewise quadratic control-Lyapunov function that can be expressed as the pointwise minimum of a finite number of quadratic functions.
\end{abstract}
\IEEEpeerreviewmaketitle

\section{Introduction}
This paper studies switching stabilization problems for continuous-time switched linear systems (SLSs). The problem is regarded as one of the basic problems in switched systems~\cite{liberzon2003switching} and has received considerable research attention in recent years~\cite{LA07,HML08,LA09}.

Existing works in this area mostly focus on deriving sufficient conditions for switching stabilizability. These conditions often guarantee the existence of certain forms of control-Lyapunov functions (CLFs). Examples include quadratic CLFs~\cite{WPD98,SESP99,liberzon2003switching}, piecewise quadratic CLFs~\cite{LA09,P03}, and composite CLFs that are obtained by taking the pointwise min, or pointwise max, or convex hull of a finite number of quadratic functions~\cite{HML08}. Once a CLF is found, a stabilizing switching law can be constructed accordingly. Despite the extensive results on sufficient conditions, establishing effective necessary conditions for switching stabilizability remains an open problem of fundamental importance. 

To establish necessary conditions, it is important to note that switching stabilizability can be defined in many ways depending on the assumptions on the switching control input~$\sigma$. One can require~$\sigma$ to be piecewise constant~\cite{liberzon2003switching,DBPL00}, or to have an average or minimum dwell time bigger than some threshold value~\cite{HM99,LM99}, or to be generated by a state-feedback switching law~\cite{LA09,HML08}. Even among the cases using feedback switching laws, switching stabilizability depends further on the solution notion used to define closed-loop trajectories, such as classical solution, Caratheodory solution, Filippov solution, or sample-and-hold solution~\cite{C09}. Therefore, the study of switching stabilizability depends crucially on the assumptions on the admissible switching input and the adopted solution notion. Unfortunately, the complication arising from different definitions of switching stabilizability has not been adequately studied in the literature.

In this paper, we consider four types of switching stabilizability. We first define the most general {\em switching stabilizability} as the existence of a measurable switching signal under which the resulting time-varying system is asymptotically stable. {\em Discrete switching stabilizability} is then defined by admitting only piecewise constant signals with switching intervals of uniform length. On the other hand, we also consider switching stabilizability under state-feedback switching laws. We call a SLS {\em feedback stabilizable in Filippov sense} (resp. {\em sample-and-hold sense}) if there exists a feedback law under which closed-loop Filippov solution (resp. sample-and-hold solution) is asymptotically stable.

This paper introduces and studies all the four switching stabilizability definitions. The main contribution is the equivalence of the following statements for a continuous-time SLS:
\begin{enumerate}[label=(\roman*)]
\item The system is switching stabilizable;
\item The system is feedback stabilizable in Filippov sense;
\item The system is feedback stabilizable in sample-and-hold sense with bounded sampling rate;
\item The system is discrete switching stabilizable;
\item There exists a piecewise quadratic CLF that can be expressed as the pointwise minimum of a finite number of quadratic functions.
\end{enumerate}

The above result represents a significant contribution to the field of switched systems. Most existing works focus on feedback stabilization in Filippov sense~\cite{WPD94,WPD98,SESP99,DBPL00,PL01,P03,LA07,HML08,LA09}. They only provide sufficient conditions, some of which even need to exclude sliding motions~\cite{WPD94,WPD98,PL01,P03,LA07}. In fact, sufficient and necessary conditions are not available even for the well studied feedback stabilization problems in Filippov sense, not to mention other types of switching stabilization problems. In contrast, we prove a unified sufficient and necessary condition for all the four switching stabilizability definitions. The result provides a fundamental insight that the class of piecewise quadratic CLFs is sufficiently rich to study switching stabilization problems under various assumptions on the switching control input. It justifies many existing works that have adopted quadratic or piecewise quadratic CLFs for simplicity or heuristic reasons~\cite{WPD94,WPD98,HB98,SESP99,DBPL00,P03,HML08,LA09}. 

It is worth pointing out an interesting connection of our result with a well known result for SLSs, namely, stability under arbitrary switching implies the existence of a common Lyapunov function that can be expressed as the {\em pointwise maximum} of a finite number of quadratic functions~\cite{MP89,LM99,liberzon2003switching}. Note that for stability analysis under arbitrary switching, the switching input can be viewed as a disturbance that tries to destabilize the system, while for switching stabilization problem, the switching input is the control that tries to stabilize the system. By our result, all the four types of switching stabilizability guarantee the existence of a {\em pointwise minimum} piecewise quadratic CLF. It is interesting to see that ``pointwise maximum'' is changed to ``pointwise minimum'' as the role of the switching input is changed from disturbance to control.

In addition, our result is also related to the classical works on nonsmooth feedback stabilization of nonlinear control systems~\cite{S83,SS95,sontag1996general,CLSS97,clarke1998nonsmooth,sontag1999stability,G00,clarke2004lyapunov,C09,C10}. The general definition of switching stabilizability is equivalent to the classical concept of asymptotic controllability~\cite{CLSS97,S83}, if we view the switching signal as a control input to the system. It is well known that asymptotic controllability is equivalent to feedback stabilizability in sample-and-hold sense for nonlinear control systems~\cite{CLSS97}. Unfortunately, such a result cannot be applied to switched systems as the open-loop vector field is required to be continuous in control~\cite{CLSS97}. In fact, the relation of switching stabilizability and feedback stabilizability in sample-and-hold sense has not been studied for switched systems. We show a stronger version of the equivalence between the two for SLSs where the sampling rate is bounded.

Furthermore, the equivalence of i) and iv) suggests that if a SLS is switching stabilizable, it can be stabilized in discrete-time. Although the result seems to be natural, its proof is highly nontrivial due to the discontinuities of the switched vector field and the weak assumption that only requires the stabilizing switching signal to be measurable. In fact, the equivalence between continuous-time and discrete-time switching stabilizability does not hold for general switched nonlinear systems. It is interesting to note that there is a counterpart for this pair of equivalence in the context of nonlinear control systems, namely, a semilinear system is asymptotic controllable if and only if it is exponentially stabilizable by a discrete feedback~\cite{G98}.


In addition to the main results, this paper provides a new perspective to study switching stabilization problems using the embedding principal~\cite{BJ95,BD05,R06}, which was originally proposed to solve switched optimal control problems. Its application to switching stabilization is new and allows us to take advantage of the numerous existing results on classical nonlinear control systems, which cannot be directly applied to switched systems. In this paper, we prove a new version of the chattering lemma~\cite{BD05} with a stronger error bound that is important for switching stabilization problems. The new chattering lemma and the idea of using embedding principal to study switching stabilization problems represent important contributions on their own.



The rest of the paper is organized as follows: In Section~II, we introduce the four switching stabilizability definitions. In Section~III, we first prove an improved chattering lemma and then establish connections between continuous-time and discrete-time switching stabilizability. In Section~IV, we prove a converse CLF theorem for the most general switching stabilizability definition. In Section~V, we show the equivalence of the four switching stabilizability definitions and develop a sufficient and necessary condition for all of them.

\textbf{Notations:} Let $\R_+$ be the set of nonnegative real numbers, $\R^n$ be the $n$-dimensional Euclidean space. Denoted by $\N$ the set of natural numbers. Denoted by $|\cdot|$ the cardinality of a given set, and $\|\cdot\|$ the Euclidean norm of a given vector or matrix. Let $\mu$ be the Lebesgue measure. 

\section{Switching Stabilizability Definitions}
In this paper, we consider the following continuous-time switched linear system (SLS):
\begin{align}
\dot{x}(t)=A_{\sigma(t)}x(t), \quad \sigma(t)\in \Q \triangleq \{1,\cdots,M\}, \label{eq:ol}
\end{align}
where $x(t)\in\R^n$ denotes the continuous state of the system, $\sigma(t)$ denotes the switching control signal that determines the active subsystem at time $t\in\R_+$, and $\{A_i\}_{i\in\Q}$ are constant matrices. Note that for any measurable switching signal $\sigma:\R_+\to\Q$, the overall switched vector field, $f(t,x(t))\triangleq A_{\sigma(t)}x(t)$, is time-varying and continuous in state $x(t)$, for which a Caratheodory solution always exists~\cite[Proposition S1]{C09}. We denote $x(\cdot;z,\sigma):\R_+\to\R^n$ as a Caratheodory solution of system~(\ref{eq:ol}) under a measurable switching signal $\sigma$ with initial state $z\in\R^n$.

The study of switching stabilizability depends crucially on the assumptions on the switching input. The switching input can be restricted to certain class of time-domain signals, or can be generated by certain class of state-feedback laws. We consider both cases in the paper. Let $\S_m$ be the set of measurable switching signals, $\S_p$ be the set of piecewise constant switching signals. Denoted by $\S_p[\tau_D]$ the set of switching signals with interval between consecutive discontinuities no smaller than $\tau_D$. Let $\S_p^+\triangleq\cup_{\tau_D\in\R_+}\S_p[\tau_D]$.
The most general definition of switching stabilizability is defined on the set of measurable switching signals $\S_m$.
\begin{defn}[Switching Stabilizability]
\label{def:ss}
System~(\ref{eq:ol}) is called {\em switching stabilizable} if for each $\epsilon>0$, there exists a $\delta>0$ such that whenever $\|z\|<\delta$, there exists a measurable $\sigma\in\S_m$ under which the state trajectory $x(\cdot;z,\sigma)$ satisfies $\|x(t;z,\sigma)\|<\epsilon$, for all $t\in\R_+$ and $x(t;z,\sigma)\to 0$ as $t\to\infty$.
\end{defn}

Definition~\ref{def:ss} is very general in the sense that it considers all measurable switching signals. In fact, if we view the switching signal $\sigma:\R_+\to\Q$ as a control input, then the switching stabilizability in Definition~\ref{def:ss} is equivalent to the classical concept of \emph{asymptotic controllability}~\cite{S83,CLSS97} for nonlinear control systems. 

If we focus on state-feedback switching laws, the definition of switching stabilizability depends further on the adopted solution notion of the closed-loop system. Assume that the state $x(t)$ is available at all time $t\in\R_+$, and the switching control is determined through a state-feedback switching law $\nu:\R^n\to\Q$. Then the corresponding closed-loop system can be written as 
\begin{align}
\dot{x}(t)=A_{\nu(x(t))}x(t). \label{eq:cl}
\end{align}

Although each subsystem vector field is continuous, the switching law $\nu$ may introduce discontinuities in the closed-loop vector field. In general, the differential equation~(\ref{eq:cl}) may not have a classical or Caratheodory solution~\cite{C09}. Filippov solution notion~\cite{filippov1988differential} is often adopted to handle the discontinuities on the right hand side of~(\ref{eq:cl}) by introducing the concept of Filippov set-valued map.

\begin{defn}[Filippov Set-Valued Map~\cite{C09}]
\label{def:filippovmap}
For any vector field $X: \mathbb{R}^n \to \mathbb{R}^n$, the corresponding Filippov set-valued map $F[X]:\mathbb{R}^n \to \mathfrak{B}(\mathbb{R}^n)$ is defined as 
\begin{align*}
F[X](x)\triangleq \bigcap_{\delta>0} \bigcap_{\mu(S)=0} \co{X(\mathcal{N}(x;\delta)\backslash S)}, \quad x \in \mathbb{R}^n,
\end{align*}
where $\mathfrak{B}(\mathbb{R}^n)$ denotes the collection of subsets of $\mathbb{R}^n$, $\overline{co}$ denotes convex closure and $\mu$ denotes the Lebesgue measure. 
\end{defn}

\begin{defn}[Filippov Solution~\cite{C09}]
A Filippov solution to a differential equation $\dot{x}(t)=X(x(t))$ over $[0,t_1]$ with $t_1>0$ is an absolutely continuous map $x:[0,t_1]\to \mathbb{R}^n$ that satisfies the differential inclusion $\dot{x}(t)\in F[X](x(t))$ for almost all $t \in [0,t_1]$.
\end{defn}

For cases where the vector field $X$ is continuous, the Filippov solution of $\dot{x}(t)=X(x(t))$ coincides with the classical solution. Whereas for cases where the vector field $X$ is discontinuous, a Filippov solution exists as long as the map $X:\mathbb{R}^n \to \mathbb{R}^n$ is measurable and locally essentially bounded~\cite{filippov1988differential}. Since each subsystem vector field is continuous, it can be easily verified that a Filippov solution to the closed-loop system~(\ref{eq:cl}) exists whenever the switching law $\nu:\mathbb{R}^n \to \mathcal{Q}$ is measurable. We denote $x(\cdot;z,\nu):\R_+\to\R^n$ as a Filippov solution of the closed-loop system~(\ref{eq:cl}) under a measurable switching law $\nu$ with initial state $z\in\R^n$. Switching stabilizability can also be defined as the existence of a switching law under which the closed-loop system is asymptotically stable in the Filippov sense.

\begin{defn}[Feedback Stabilizability in Filippov Sense]
\label{def:fs}
System~(\ref{eq:ol}) is called {\em feedback stabilizable in Filippov sense} if $\exists$ a measurable switching law $\nu:\R^n\to\Q$ such that for each $\epsilon>0$, there exists a $\delta>0$ for which whenever $\|z\|<\delta$, any closed-loop Filippov trajectory $x(\cdot;z,\nu)$ satisfies that $\|x(t;z,\nu)\|<\epsilon, \forall t\in\R_+$, and $x(t;z,\nu)\to 0$ as $t\to\infty$.
\end{defn}

Definition~\ref{def:fs} is very useful for switching stabilization problems due to the crucial importance of Filippov solution to switched systems. In fact, most existing studies on switching stabilization adopt Definition~\ref{def:fs} to derive various sufficient conditions for switching stabilizability~\cite{WPD98,P03,HML08}. 

Sample-and-hold (abbrev. S-H) solution (or $\pi$-solution) is another widely used solution notion for discontinuous dynamical systems~\cite{C09,CLSS97}. Any infinite sequence $\pi=\{t_k\}_{k\in\N}$ where $0=t_0<t_1<t_2<\cdots$ and $t_k\to\infty$ as $k\to\infty$ is called a {\em sampling schedule}, and the number $d(\pi)\triangleq\sup\{t_{k+1}-t_k, k\in\N\}$ is called the {\em diameter} of schedule $\pi$.
\begin{defn}[Sample-and-Hold Solution~\cite{CLSS97}]
Let $X:\R^n\times \mathcal{U}\to\R^n$. Given a feedback law $\nu:\R^n\to\mathcal{U}$, an initial condition $z$ and a sampling schedule $\pi=\{t_k\}_{k\in\N}$, a {\em sample-and-hold solution} (or {\em $\pi$-solution}) of $\dot{x}(t)=X(x(t),\nu(x(t)))$ is the map $x:\R_+\to\R^n$, with $x(0)=z$, defined recursively by solving $\dot{x}(t)=X(x(t),\nu(x(t_k))),t\in[t_k,t_{k+1}]$ for $k\in\N$.
\end{defn}
The existence of S-H solution is guaranteed if for all $u\in\mathcal{U}$, the map $x\mapsto X(x,u)$ is continuous~\cite{C09}. We denote $x_\pi(\cdot;z,\nu)$ as the $\pi$-solution of the closed-loop system~(\ref{eq:cl}) under a measurable switching law $\nu$ with initial state $z\in\R^n$. One may interpret S-H solution as representing the behavior of sampling under a fixed feedback law. The feedback control is evaluated only at sampling times with the values being held until the next sampling time. Feedback stabilizability in the context of S-H solution means asymptotic stability of the sampled closed-loop system, which in general may involve an unbounded sampling rate as the trajectory approaches to the origin. In this paper, we are interested in the case where asymptotic stability can be obtained by sampling with bounded rate (i.e. nonvanishing intersampling time).

\begin{defn}[Feedback Stabilizability in S-H Sense with Bounded Sampling Rate]
\label{def:fssh}
System~(\ref{eq:ol}) is called {\em feedback stabilizable in S-H sense with bounded sampling rate} if there exists a feedback law $\nu:\R^n\to\Q$ and a constant $h_0>0$ such that whenever $d(\pi)<h_0$, the closed-loop $\pi$-solution $x_\pi(\cdot;z,\nu)$ satisfies $\forall\epsilon>0,\exists \delta>0$ such that whenever $\|z\|<\delta, \|x_\pi(t;z,\nu)\|<\epsilon,\forall t\in\R_+$ and $x_\pi(t;z,\nu)\to 0$ as $t\to\infty$.
\end{defn}


\begin{rem} 
The traditional definition of feedback stabilizability in S-H sense~\cite[Definition~I.3]{CLSS97} does not require a uniformly bounded sampling rate. In particular, the bound of sampling rate $1/h=1/h(\delta,\epsilon)$ used in the definition of the ``s-stabilizing feedback'' in~\cite{CLSS97} depends on both the initial state radius $\delta$ and the region of attraction radius $\epsilon$. As a result, the sampling rate may grow unbounded as the state converges to the origin. Definition~\ref{def:fssh} requires the existence of a uniform bound on the sampling rate (i.e. $1/h_0<\infty$) for all pairs of $(\delta,\epsilon)$ and thus is stronger than the traditional definition.
\end{rem}

Switching stabilizability defined in Definition~\ref{def:fssh} clearly implies the existence of a piecewise constant stabilizing signal $\sigma\in\S_p[h]$ for all $h\in(0,h_0)$. This is different, but closely related to the discrete switching stabilizability defined below.

\begin{defn}[Discrete Switching Stabilizability] 
\label{def:dss}
System~(\ref{eq:ol}) is called {\em discrete switching stabilizable} if there exists a constant $h_0>0$ such that for any $h\in(0,h_0)$, there exists a $\sigma:\R_+\to\Q$ with $\sigma(t)=\sigma_k\in\Q,\forall t\in[kh,(k+1)h),\forall k\in\N$ under which the state trajectory $x(\cdot;z,\sigma)$ satisfies $\forall\epsilon>0,\exists \delta>0$ such that $\|z\|<\delta$ implies that $\|x(t;z,\sigma)\|<\epsilon,\forall t\in\R_+$ and $x(t;z,\sigma)\to 0$ as $t\to\infty$.
\end{defn}
We call the stabilizability in Definition~\ref{def:ss}, \ref{def:fs}, \ref{def:fssh} and \ref{def:dss} {\em exponential} if there exists $C>0,\gamma>0$ such that the closed-loop solution $x(\cdot):\R_+\to\R^n$ with $x(0)=z$ satisfies $\|x(t)\|\le Ce^{-\gamma t}\|z\|,\forall t\in\R_+,\forall z\in\R^n$.

The goal of this paper is to show the four switching stabilizability definitions are all equivalent to the existence of a piecewise quadratic control-Lyapunov function (CLF). Furthermore, such a CLF can be expressed as the pointwise minimum of a finite number of quadratic functions.

The readers are referred to the introduction section for the significance and challenges for establishing these results. In the rest of the paper, we will prove the main results in three steps. First, we establish connections between continuous-time and discrete-time switching stabilizability (Section~III). Then, we use some converse results for switching stabilization in discrete-time to construct piecewise quadratic CLFs (Section~IV). Lastly, we show the equivalence of the four switching stabilizability definitions and the existence of piecewise quadratic CLF (Section~V).

\section{Connection to Discrete-Time Switching Stabilizability}
In this section, we will establish connections between continuous-time and discrete-time switching stabilizability. The goal is to show that the general switching stabilizability defined in Definition~\ref{def:ss} implies exponential discrete switching stabilizability (Definition~\ref{def:dss}).

It is well known that asymptotic controllability implies feedback stabilizability in S-H sense for general nonlinear control systems~\cite{CLSS97}. However, such a result cannot be directly applied to switched systems as the open-loop vector field is required to be continuous in control in~\cite{CLSS97}. In fact, even if we have such a result for switched systems, it still does not imply discrete switching stabilizability due to the possibly unbounded growth of the sampling rate close to the origin. As a result, the intersampling time will vanish and the corresponding discrete-time system is not well defined. Therefore, nonvanishing intersampling time is essential for establishing the connection to discrete-time switching stabilization problems.

In general, intersampling time has to tend to zero to stabilize the sampled closed-loop system. One exception is homogeneous system whose open-loop vector field satisfies $g(ax,u)=ag(x,u),\forall a\ge 0$. For such systems, it is shown in~\cite{G00} that asymptotic controllability implies feedback stabilizability in S-H sense with bounded sampling rate. However, the result cannot be directly applied here as $g$ is required to be continuous in both $x$ and $u$ in~\cite{G00}, while the open-loop vector field of system~(\ref{eq:ol}) is not continuous in $\sigma$. To deal with the discontinuities due to the switching control $\sigma$, we first introduce and study a relaxed system that is continuous in control.

\subsection{Relaxed System and Improved Chattering Lemma}
Embedding principal is a well known approach for solving switched optimal control problems~\cite{BJ95,BD05,R06}. It embeds switched system into a larger family of nonlinear systems with relaxed continuous control inputs so that the set of trajectories of the original switched system is dense in those of the relaxed system. It is shown by the so-called chattering lemma that trajectories of the relaxed system can be approximated by those of the switched system with error bound of arbitrary accuracy by proper choice of {\em measurable} switching signal. Our derivation of the connection between continuous-time and discrete-time switching stabilizability is also based on the embedding principal. It turns out that we require an error bound that is much stronger than the one provided by the chattering lemma in~\cite{BD05}. In this subsection, we will prove a new chattering lemma for switching stabilization problems. 

Denote $\mathcal{U}_p\triangleq \{\alpha\in\{0,1\}^M: \sum_{i=1}^M \alpha_i=1\}$ and $\mathcal{U}_r\triangleq \{\alpha\in[0,1]^M: \sum_{i=1}^M \alpha_i=1\}$. We refer to system~(\ref{eq:ol}) as a pure system $(\mathcal{P})$, which can be equivalently written as
\begin{align*}
(\mathcal{P}): \dot{x}(t)=\sum_{i\in\Q} \alpha_i(t)A_ix(t),\quad \alpha(t)\in\mathcal{U}_p.
\end{align*}
Define the corresponding relaxed system $(\mathcal{R})$ as
\begin{align*}
(\mathcal{R}): \dot{x}(t)=\sum_{i\in\Q} \alpha_i(t)A_ix(t),\quad \alpha(t)\in\mathcal{U}_r.
\end{align*}
Let $x(\cdot;z,\alpha^p):\R_+\to\R^n$ be the state trajectory of $(\mathcal{P})$ under a pure control signal $\alpha^p:\R_+\to \mathcal{U}_p$ and $x(\cdot;z,\alpha^r):\R_+\to\R^n$ be the state trajectory of $(\mathcal{R})$ under a relaxed control signal $\alpha^r:\R_+\to\mathcal{U}_r$. We call a relaxed control signal $\alpha^r:\R_+\to\mathcal{U}_r$ {\em exponentially stabilizing} if $\exists C,\gamma>0$ s.t. $\|x(t;z,\alpha^r)\|\le Ce^{-\gamma t}\|z\|,\forall t\in\R_+,\forall z\in\R^n$. The new chattering lemma proves an error bound proportional to the norm of initial state.

\begin{lem}
\label{lem:chattering}
For any exponentially stabilizing relaxed control signal $\alpha^r:[0,T]\to\mathcal{U}_r$ and any $\epsilon>0$, there exists a pure control signal $\alpha^p:[0,T]\to\mathcal{U}_p$ where $\alpha^p\in\S_p^+$ such that $\|x(t;z,\alpha^p)-x(t;z,\alpha^r)\|<\epsilon \|z\|,\forall t\in[0,T],\forall z\in\R^n$.
\end{lem}
\begin{proof}
Denote $\phi(t)\triangleq x(t;z,\alpha^p)$ and $\tilde\phi(t)\triangleq x(t;z,\alpha^r)$. Given relaxed control signal $\alpha^r:[0,T]\to\mathcal{U}_r, \epsilon>0$ and initial state $z\in\R^n$, the goal is to construct a pure control signal $\alpha^p:[0,T]\to\mathcal{U}_p$ where $\alpha^p\in\S_p^+$ such that $\|\phi(t)-\tilde\phi(t)\|<\epsilon \|z\|,\forall t\in [0,T]$. We first partition $[0,T]$ into equal length subintervals and then apply the following construction strategy for each subinterval. Let $h>0$ be the length of subinterval (we will decide its upper bound later). On each subinterval $[kh,(k+1)h),k\in\N$, $\alpha^p$ sequentially takes value from the set $\mathcal{U}_p$ of $M$ elements, i.e.,
\begin{align}
\alpha^p_i(t)=\left\{\begin{array}{ll} 1, & t \in [t_{k,i-1},t_{k,i}) \\ 0, & \text{otherwise} \end{array}\right., \forall i=1,\cdots,M, \label{eq:purecontrol}
\end{align}
where $t_{k,0}=kh$ and $t_{k,i}$ are defined recursively by 
\begin{align}
t_{k,i}=t_{k,i-1}+\int_{kh}^{(k+1)h} \alpha^r_i(\tau)d\tau, \forall i=1,\cdots,M. \label{eq:tki}
\end{align}
By construction, $\Delta t_{k,i}\triangleq t_{k,i}-t_{k,i-1}>0,\forall k\in\N,i\in\Q$ and thus $\alpha^p\in\S_p^+$. Similar as the proof in~\cite{BD05}, the error can be divided into two terms, i.e., $\|\phi(t)-\tilde{\phi}(t)\| \le E_1+E_2$, where 
\begin{align*}
& E_1\triangleq\Big\|\int_0^t \sum_{i=1}^M \alpha^p_i(\tau)A_i\big(\phi(\tau)-\tilde{\phi}(\tau)\big) d\tau\Big\|, \\ 
& E_2\triangleq\Big\|\int_0^t \sum_{i=1}^M \big(\alpha^p_i(\tau)-\alpha^r_i(\tau)\big)A_i \tilde{\phi}(\tau) d\tau\Big\|. 
\end{align*}
Next, we derive the upper bounds for $E_1$ and $E_2$. By matrix norm inequality and $\alpha^p\in\mathcal{U}_p$,
\begin{align*}
E_1 &\le \int_0^t \sum_{i=1}^M \|\alpha^p_i(\tau)A_i\big(\phi(\tau)-\tilde{\phi}(\tau)\big)\| d\tau \\ 
&\le \int_0^t \sum_{i=1}^M \alpha^p_i(\tau) \|A_i\| \|\phi(\tau)-\tilde{\phi}(\tau)\| d\tau \\ 
&\le L_1 \int_0^t\|\phi(\tau)-\tilde{\phi}(\tau)\|d\tau, \text{ where } L_1\triangleq \max_{i\in\Q}\|A_i\|.
\end{align*}
Due to the construction of $\alpha^p$ in~(\ref{eq:purecontrol}), we have i) $\int_{kh}^{(k+1)h} \sum_{i=1}^M \alpha^p_i(\tau)A_i \tilde{\phi}(\tau)d\tau=\sum_{i=1}^M \int_{t_{k,i-1}}^{t_{k,i}} A_i\tilde{\phi}(\tau) d\tau$. It follows from~(\ref{eq:tki}) that ii) $\int_{[t_{k,i-1},t_{k,i})}(1-\alpha^r_i(\tau))d\tau=\int_{[kh,(k+1)h)\backslash [t_{k,i-1},t_{k,i})}\alpha^r_i(\tau)d\tau$. Let $\tilde{\phi}^\Delta(t)\triangleq \tilde{\phi}(t)-\tilde{\phi}(t_k)$. Since $\alpha^r$ is exponentially stabilizing, $\exists C>0$ s.t. $\|\tilde{\phi}(t)\|\le C\|z\|,\forall t\in\R_+$ and thus iii) $\|\tilde{\phi}^\Delta(t)\|\le hL_1C\|z\|,\forall t\in[kh,(k+1)h)$. Based on i), ii) and iii),
\begin{align*}
E_2 \le& \sum_k\Big\|\int_{kh}^{(k+1)h} \sum_{i=1}^M \big(\alpha^p_i(\tau)A_i \tilde{\phi}(\tau)-\alpha^r_i(\tau)A_i \tilde{\phi}(\tau)\big) d\tau\Big\| \\
\le& \sum_k\sum_{i=1}^M \Big\| \int_{[t_{k,i-1},t_{k,i})} (1-\alpha^r_i(\tau))A_i\tilde{\phi}(\tau) d\tau - \\ &\qquad\qquad \int_{[kh,(k+1)h)\backslash [t_{k,i-1},t_{k,i})}\alpha^r_i(\tau)A_i\tilde{\phi}(\tau) d\tau\Big\|\\
\le& \sum_k\sum_{i=1}^M \int_{kh}^{(k+1)h} \|A_i\| \|\tilde{\phi}^\Delta(\tau)\|d\tau\le \frac{T}{h}Mh^2L_1^2C\|z\|. 
\end{align*}
Let $\kappa\triangleq TML_1^2C$. By choosing $h<\frac{\epsilon}{\kappa}e^{-L_1T}$, the rest of the proof follows from Gronwall inequality.
\end{proof}

\begin{rem}
The new chattering lemma differs from the original version in the following aspects: i) The error bound is any desired accuracy times the norm of initial state rather than just the desired accuracy; ii) The choice of switching signals is from the set $\S_p^+$ rather than the set $\S_m$; iii) It is under the assumption of relaxed control signal being exponentially stabilizing. In fact, the above three properties play important roles in establishing the connection to exponential discrete switching stabilizability.
\end{rem}

Lemma~\ref{lem:chattering} indicates that the set of trajectories of system~(\ref{eq:ol}) under switching signals from the set $\S_p^+$ is dense in the set of exponentially stable trajectories of the relaxed system $(\mathcal{R})$. It allows us to only focus on signals from the set $\S_p^+$ in the approximation of any exponentially stable trajectory of $(\mathcal{R})$.

\subsection{Switching Stabilizability Implies Discrete Switching Stabilizability}
The relaxed system $(\mathcal{R})$ is a homogeneous system, whose vector field is continuous with respect to both state and the control input $\alpha^r$. System $(\mathcal{R})$ is called {\em asymptotically controllable} if for each $\epsilon>0$, there exists a $\delta>0$ such that whenever $\|z\|<\delta$, there exists a control $\alpha^r:\R_+\to\mathcal{U}_r$ under which the state trajectory $x(\cdot;z,\alpha^r)$ satisfies $\|x(t;z,\alpha^r)\|<\epsilon$, for all $t\in\R_+$ and $x(t;z,\alpha^r)\to 0$ as $t\to\infty$. An important property of asymptotic controllability for homogeneous systems is stated below.
\begin{lem}[Proposition 4.4~\cite{G00}]
\label{lem:homo}
If system $(\mathcal{R})$ is asymptotically controllable, then it is exponentially feedback stabilizable in S-H sense with bounded sampling rate, i.e., there exists a feedback law $\nu:\R^n\to\mathcal{U}_r$ and constants $h_0>0,C>0,\gamma>0$ such that any closed-loop $\pi$-solution with $d(\pi)<h_0$ satisfies $\|x_\pi(t;z,\nu)\|\le Ce^{-\gamma t}\|z\|,\forall t\in\R_+,\forall z\in\R^n$.
\end{lem}

By Lemma~\ref{lem:homo}, asymptotic controllability of the relaxed system $(\mathcal{R})$ guarantees exponential stability of the sampled closed-loop system with sufficiently small but nonvanishing intersampling time. The existence of exponentially stable trajectories of $(\mathcal{R})$ allows us to construct exponentially stabilizing switching signals from the set $\S_p^+$ based on Lemma~\ref{lem:chattering}.

\begin{lem}
\label{lem:sspwconst}
If system~(\ref{eq:ol}) is switching stabilizable, then it is exponentially switching stabilizable under a switching signal $\sigma:\R_+\to\Q$ where $\sigma\in\S_p^+$.
\end{lem}
\begin{proof}
Consider the pure system $(\mathcal{P})$ and the relaxed system $(\mathcal{R})$ defined before. Obviously, $(\mathcal{R})$ is asymptotically controllable given $(\mathcal{P})$ is switching stabilizable. Furthermore, $(\mathcal{R})$ is exponentially feedback stabilizable in S-H sense with bounded sampling rate by Lemma~\ref{lem:homo}. Let $\nu:\R^n\to\mathcal{U}_r$ be the stabilizing feedback law of $(\mathcal{R})$. The goal is to find an exponentially stabilizing signal $\sigma\in\S_p^+$ of $(\mathcal{P})$. 

We now fix a nonvanishing sampling schedule $\pi=\{t_k\}_{k\in\N}$ and consider a relaxed control signal defined as $\alpha^r(t)=\nu(x_{\pi}(t_k;z,\nu)), \forall t\in [t_k,t_{k+1}),\forall k\in\N$. As $\alpha^r$ is exponentially stabilizing, there exists $C>0,\gamma>0$ such that $\|x(t;z,\alpha^r)\|\le Ce^{-\gamma t}\|z\|,\forall z\in\R^n,\forall t\in\R_+$. Let the finite horizon $T>\frac{2\log(2C)}{\gamma}$ and $\epsilon=Ce^{-\gamma T}$. By Lemma~\ref{lem:chattering}, we can construct a pure control signal $\alpha^{p,0}:[0,T]\to\mathcal{U}_p$ where $\alpha^{p,0}\in\S_p^+$ such that $\|x(t;z,\alpha^{p,0})-x(t;z,\alpha^r)\|<\epsilon\|z\|,\forall t\in[0,T],\forall z\in\R^n$. Then, the state trajectory of $(\mathcal{P})$ under $\alpha^{p,0}$ satisfies $\|x(t;z,\alpha^{p,0})\|\le \|x(t;z,\alpha^r)\|+\|x(t;z,\alpha^{p,0})-x(t;z,\alpha^r)\| \le Ce^{-\gamma t}\|z\|+\epsilon\|z\| \le 2Ce^{-\gamma t}\|z\|,\forall t\in[0,T]$. We next iteratively apply the bound on intervals of length $T$ to obtain the exponential convergence on $\R_+$. Let $\alpha^p:\R_+\to\mathcal{U}_p$ be the concatenation of $\alpha^{p,k}:[kT,(k+1)T)\to\mathcal{U}_p, k\in\N$. For $t\in[kT,(k+1)T)$, $\|x(t;z,\alpha^p)\|\le (2C)^{k+1}e^{-\gamma kT}\|z\|<e^{-(\frac{k}{k+1}\gamma-\frac{\log(2C)}{T})t}\|z\|$. In general, for any $t\in\R_+$, $\|x(t;z,\alpha^p)\|<e^{-\gamma' t}\|z\|$ where $\gamma'=\frac{\gamma}{2}-\frac{\log(2C)}{T}\in(0,\gamma)$. The stabilizing signal $\sigma:\R_+\to\Q$ can be obtained from $\alpha^p$ as follows: $\sigma(t)=i$, if $\alpha^p_i(t)=1,\forall t\in\R_+$. As $\alpha^{p,k}\in\S_p^+,\forall k\in\N$, we have $\alpha^p\in\S_p^+$ and thus $\sigma\in\S_p^+$.
\end{proof}

Now we have found a switching signal $\sigma\in\S_p^+$ that exponentially stabilizes the system. However, the stabilizing switching signal $\sigma$ may not have a uniform intersampling time. It remains to show that if we sample the signal with a fixed intersampling time that is sufficiently small and hold the signal until the next sampling, the corresponding state trajectory is also exponentially stable. This will then imply discrete switching stabilizability.
\begin{thm}
\label{thm:sstodss}
If system~(\ref{eq:ol}) is switching stabilizable, then it is exponentially discrete switching stabilizable.
\end{thm}
\begin{proof}
By Lemma~\ref{lem:sspwconst}, there exists a switching signal $\sigma_0:\R_+\to\Q$ where $\sigma_0\in\S_p^+$ under which the state trajectory $x(\cdot;z,\sigma_0)$ is exponentially stable, i.e., $\exists C_0>0,\gamma>0$, s.t. $\|x(t;z,\sigma_0)\|\le C_0e^{-\gamma t}\|z\|,\forall t\in\R_+,\forall z\in\R^n$. Let $\sigma_h:\R_+\to\Q$ be the sampled signal of $\sigma_0$ with sampling intervals of uniform length $h$, i.e., $\sigma_h(t)\triangleq \sigma_0(kh),\forall t\in [kh,(k+1)h),\forall k\in\N$. Let $\phi_0(t)\triangleq x(t;z,\sigma_0), \phi_h(t)\triangleq x(t;z,\sigma_h)$. The rest of the proof has two ingredients: i) the exponential convergence of the error between $\phi_0$ and $\phi_h$ on a finite horizon and ii) the extension of the exponential convergence of $\phi_h$ from a finite horizon to $\R_+$.

To show i), one can follow the proof of Lemma~\ref{lem:chattering} by dividing the error into two terms and bounding the first term with integral of the error and the second term with constant times $h\|z\|$. Let $L_1\triangleq \max_{i\in\Q}\|A_i\|,L_2\triangleq \max_{i,j\in\Q}\|A_i-A_j\|$. For the first term, $E_1\triangleq \|\int_0^t A_{\sigma_h(\tau)}\big(\phi_h(\tau)-\phi_0(\tau)\big)d\tau\|\le L_1\int_0^t \|\phi_h(\tau)-\phi_0(\tau)\|d\tau$. We now discuss the second term. Since $\sigma_0\in\S_p^+$, there are at most $N<\infty$ switches on a finite interval and thus i.1) $\sigma_h$ and $\sigma_0$ differ on intervals of length at most $Nh$. As $\sigma_0$ is exponentially stabilizing, i.2) $\|\phi_0(t)\| \le C_0\|z\|,\forall t\in \R_+$. Based on i.1) and i.2), $E_2\triangleq \int_0^t \|\big(A_{\sigma_h(\tau)}-A_{\sigma_0(\tau)}\big)\phi_0(\tau)\|d\tau\le L_2 N h C_0\|z\|\triangleq \kappa h\|z\|$. Let $h_0\triangleq \frac{C_0}{\kappa}e^{-(L_1+\gamma)T}$. By choosing $h\in(0,h_0)$, the rest of i) follows from Gronwall inequality.

To show ii), one can follow the proof of Lemma~\ref{lem:sspwconst} by iteratively applying the bound on intervals of length $T$. By i), for sufficiently small $h$, $\|\phi_h(t)\|\le 2C_0e^{-\gamma t}\|z\|,\forall t\in [0,T],z\in\R^n$. By choosing $T>\frac{2\log(2C_0)}{\gamma}$, $\|\phi_h(t)\|<e^{-\gamma' t}\|z\|,\forall t\in\R_+,\forall z\in\R^n$ where $\gamma'=\frac{\gamma}{2}-\frac{\log(2C_0)}{T}\in(0,\gamma)$.
\end{proof}

The above theorem indicates that switching stabilizability implies exponential switching stabilizability of discrete-time systems obtained by sampling the original system with sufficiently small and fixed intersampling time. Although such a result appears to be natural, its proof is highly nontrivial due to the possibility of wild behaviors of a measurable stabilizing switching signal $\sigma\in\S_m$ and the discontinuity of the switched vector field with respect to the switching input $\sigma$. In fact, the result does not hold for general switched nonlinear systems, for which the existence of a stabilizing switching signal $\sigma\in\S_m$ does not imply the existence of a $\sigma\in\S_p^+$ with switching intervals of uniform length.

\section{Converse Control-Lyapunov Function Theorem for Switched Linear Systems}
In this section, we will develop a converse CLF theorem for the switching stabilizability in Definition~\ref{def:ss} where the switching control $\sigma$ is only required to be measurable. This is more general than the definition used in many other works~\cite{liberzon2003switching,LA05} for SLSs that require $\sigma$ to be piecewise constant. 

\subsection{Control-Lyapunov Function}
Control-Lyapunov function (CLF) is an important tool to study stabilization problems. This paper focuses on an important class of nonsmooth CLFs, namely, pointwise minimum piecewise quadratic CLFs. 
\begin{defn}[Pointwise Minimum Piecewise Quadratic Function (pm-PQF)]
\label{def:ptmin}
Let $P_j,j\in\N_m$ be symmetric matrices, i.e., $P_j^T=P_j,\forall j\in\N_m$. The function defined by 
\begin{align}
V(x)\triangleq \min_{j\in\N_m}x^TP_jx, \quad x\in\R^n,
\end{align}
is called a pm-PQF if $\Omega_j\neq\emptyset$, $\forall j\in\N_m$, where $\Omega_j\triangleq \{x\in\R^n: x^TP_jx<x^TP_kx, \forall k\neq j\}$.
\end{defn}
The above definition ensures that every quadratic function $V_j(x)\triangleq x^TP_jx$ contributes nontrivially to the pointwise minimum. Due to the smoothness of $V_j$, each $\Omega_j$ is a full-dimensional open set in $\R^n$ with continuously differentiable boundary $\partial\Omega_j$ of measure zero. A pm-PQF is clearly a piecewise smooth function, for which directional derivative exists everywhere~\cite[p.43]{subbotin1995generalized}. 

\begin{lem}[\cite{subbotin1995generalized}]
\label{lem:dd}
For any pm-PQF $g:\mathbb{R}^n \to \mathbb{R}$, the limit $Dg(x;\eta) \triangleq \lim_{\delta \downarrow 0} \frac{1}{\delta}(g(x+\delta \eta)-g(x))$ exists, $\forall x, \eta \in \mathbb{R}^n$. 
\end{lem}
Due to Lemma~\ref{lem:dd}, directional derivative is well defined for pm-PQFs. Therefore, we can define CLFs based on pm-PQFs where conditions are given in terms of directional derivative.
\begin{defn}[Pointwise Minimum Piecewise Quadratic Control-Lyapunov Function (pm-PQCLF)] 
\label{def:ptminclf}
A pm-PQF $V:\R^n\to\R_+$ is called a pm-PQCLF if there exists a continuous function $W:\R^n\to\R_+$ such that the following conditions hold:
\begin{align}
& V(x)>0, W(x)>0, \forall x \neq 0, \quad V(0)=0; \label{cond:pd} \\
& \mathcal{L}_{\beta}=\{x:V(x) \le \beta\} \text{ is bounded for each } \beta; \label{cond:prop} \\
& \min_{i \in \mathcal{Q}} DV(x;f_i(x)) \le -W(x), \quad \forall x \in \mathbb{R}^n.  \label{cond:inf} 
\end{align}
\end{defn}
We refer to (\ref{cond:pd}) as the positive definite condition, refer to (\ref{cond:prop}) as the radially unbounded condition, and refer to (\ref{cond:inf}) as the decreasing condition. For a discrete-time SLS, it has been shown that switching stabilizability implies the existence of a pm-PQCLF~\cite{ZAHV09}. We will prove a similar converse pm-PQCLF theorem for continuous-time switching stabilizability. The proof relies on the connection between continuous-time and discrete-time switching stabilizability established in the previous section. Consider the discrete-time switched linear system (DTSLS) obtained by sampling system~(\ref{eq:ol}) with intervals of length $h$:
\begin{align}
x(k+1)=e^{A_{\sigma(k)}h}x(k), \quad \sigma(k)\in\Q,k\in\N. \label{eq:dt}
\end{align}
Denote $x^h(\cdot;z,\sigma):\N\to\R^n$ as the solution of DTSLS~(\ref{eq:dt}) under a switching sequence $\sigma:\N\to\Q$ with initial state $z\in\R^n$. As shown in~\cite{ZHA12}, pm-PQCLFs for DTSLSs can be constructed from finite-horizon value function defined below.
\begin{defn}[Value Function] 
Denoted by $J^h_N(z,\sigma)\triangleq\sum_{k=0}^N\|x^h(k;z,\sigma)\|^2$ the $N$-horizon cost function of system~(\ref{eq:dt}) with initial state $z$ under switching sequence $\sigma=\{\sigma_k\}_{k=0}^N$. The $N$-horizon value function of system~(\ref{eq:dt}) is defined as $V_N^h(z)=\min_{\sigma}J^h_N(z,\sigma)$.
\end{defn}

It can be easily shown that the value function defined above is a pm-PQF.
\begin{lem}[\cite{ZHA12}]
The N-horizon value function of system~(\ref{eq:dt}) takes the form of $V^h_N(z)=\min_{P\in\mathcal{H}_N}z^TPz$ where $\mathcal{H}_N$ is a finite set of positive definite matrices.
\end{lem}

The converse result for switching stabilizability of DTSLSs is developed in terms of finite-horizon value functions. It suggests that the finite-horizon value function $V_N^h$ will eventually become a pm-PQCLF as the horizon $N$ increases.
\begin{thm}[\cite{ZAHV09}]
\label{thm:dtconverse}
If system~(\ref{eq:dt}) is exponentially switching stabilizable, there exists constants $N_0<\infty, \kappa>0$ such that for any $N\ge N_0$, the N-horizon value function $V^h_N$ satisfies
\begin{align}
\min_{i\in\Q}\{V_N^h(e^{A_ih}z)-V_N^h(z)\} \le -\kappa\|z\|,\quad \forall z\in\R^n. \label{cond:dtinf}
\end{align}
\end{thm}
Note that condition~(\ref{cond:dtinf}) can be considered as a discrete-time version of the decreasing condition~(\ref{cond:inf}). As we will see next, condition~(\ref{cond:dtinf}) in discrete-time is the key to prove condition~(\ref{cond:inf}) in continuous-time by proper choice of $h$ and $N$.

\subsection{Converse pm-PQCLF Theorem}
We now develop a converse CLF result for the most general switching stabilizability (Definition~\ref{def:ss}). According to Theorem~\ref{thm:sstodss}, switching stabilizability implies exponential switching stabilizability of a collection of DTSLSs~(\ref{eq:dt}) with sufficiently small $h$. Then, Theorem~\ref{thm:dtconverse} ensures that the finite-horizon value function $V^h_N$ is a pm-PQCLF for DTSLS~(\ref{eq:dt}). We want to show that $V_N^h$ is also a pm-PQCLF for system~(\ref{eq:ol}). The proof uses the property of $V_N^h$ being a pm-PQF to write the difference term in condition~(\ref{cond:dtinf}) as the sum of the directional derivative term in condition~(\ref{cond:inf}) and an error term that can always be made positive by choosing small enough $h$. The dependency of $V^h_N$ on $h$ imposes some challenges for establishing the desired result.
\begin{lem} 
\label{lem:dssclf}
If system~(\ref{eq:ol}) is exponentially discrete switching stabilizable, then it admits the finite-horizon value function $V^h_N$ with sufficiently small $h$ and sufficiently large $N$ as pm-PQCLFs.
\end{lem}
\begin{proof}
Obviously, $V^h_N$ satisfies the positive definite condition~(\ref{cond:pd}) and the radially unbounded condition~(\ref{cond:prop}). We are left to show that it also satisfies the decreasing condition~(\ref{cond:inf}). By the assumption of exponential discrete switching stabilizability, there exist constants $h_0>0, C>0,\gamma>0,\kappa>0$ such that for any DTSLS~(\ref{eq:dt}) with $h\in (0,h_0)$, there exists a switching sequence $\sigma$ under which the state trajectory satisfies $\|x^h(k;z,\sigma)\|\le Ce^{-\gamma hk}\|z\|,\forall z\in\R^n,\forall k\in\N$. Furthermore, there exists a $N<\infty$ such that
\begin{align*}
\min_{i\in\Q}\{V_N^h(e^{A_ih}z)-V_N^h(z)\} \le -\kappa h\|z\|,\forall z\in\R^n.
\end{align*}
Since $V_N^h=\min_{P\in \mathcal{H}_N}z^TPz$, we have
\begin{align*}
\min_{i\in\Q}\{z^T(e^{A_ih})^TP'e^{A_ih}z-z^TPz\} \le -\kappa h\|z\|, \text{ where } \\ P\triangleq \argmin_{P\in \mathcal{H}_N}z^TPz, P'\triangleq \argmin_{P\in \mathcal{H}_N}z^T(e^{A_ih})^TPe^{A_ih}z.
\end{align*}
By Taylor expansion, $e^{A_ih}=I+A_ih+o(h^2)$, which gives
\begin{align*}
\min_{i\in\Q}\{z^T(P'-P)z+hz^T(A_i^TP'+P'A_i)z\} \le -\kappa h\|z\|.
\end{align*}
Note that the directional derivative of $V_N^h$ at $z$ takes the form of $DV_N^h(z;A_iz)=z^T(A_i^TP+PA_i)z$. Then,
\begin{multline}
\min_{i\in\Q} DV_N^h(z;A_iz)+\min_{i\in\Q} \big\{\frac{1}{h} z^T(P'-P)z+ \\ z^T\big(A_i^T(P'-P)+(P'-P)A_i\big)z\big\} \le -\kappa\|z\|. \label{eq:minDV}
\end{multline}
Let $\Delta P\triangleq P'-P$. We next discuss the order of $\|P\|$ and $\|\Delta P\|$ for their dependency on $h$. We claim that i) $\|P\|=O(\frac{1}{h})$ and ii) $\|\Delta P\|=O(1)$. The proof goes as follows: i) By monotonicity of value function in terms of horizon, $V^h_N(z)=z^TPz=O(\frac{1}{1-e^{-\gamma h}})\cdot \|z\|$. Thus, $\|P\|=O(\frac{1}{1-e^{-\gamma h}})\to O(\frac{1}{h})$ as $h\to 0$. ii) Again by monotonicity property, $|V^h_N(e^{A_ih}z)-V^h_N(z)|=|z^T(e^{A_ih})^TP'e^{A_ih}z-z^TPz|=O(\frac{1}{1-e^{-\gamma h}})\cdot \|e^{A_ih}z-z\|$. It then gives $\|(e^{A_ih})^TP'e^{A_ih}-P\|=O(\frac{1}{1-e^{-\gamma h}})\cdot O(h)$. Note that $\|(e^{A_ih})^TP'e^{A_ih}-P\|=\|\Delta P+h(A_i^TP'+P'A_i)+o(h^2)\|\ge \|\Delta P\|-h\|A_i^TP'+P'A_i\|$, where the inequality is due to matrix norm triangle inequality. By reorganizing terms, we have $\|\Delta P\|\le O(\frac{1}{1-e^{-\gamma h}})\cdot O(h)+h\|A_i^TP'+P'A_i\|=O(\frac{1}{1-e^{-\gamma h}})\cdot O(h)+O(h)\cdot O(\frac{1}{1-e^{-\gamma h}})\to O(1)$ as $h\to 0$, where the equality is due to the order of $\|P\|$ discussed in i). Based on the property of pointwise minimum that $z^T\Delta Pz\ge 0$ and the fact that $\|\Delta P\|=O(1)$ we just proved, there exists a sufficiently small $h>0$ such that $z^T(\frac{1}{h}\Delta P+A_i^T\Delta P+\Delta PA_i)z>0,\forall z\in\R^n,\forall i\in\Q$. In other words, $\min_{i\in\Q} \{\frac{1}{h} z^T\Delta Pz+z^T(A_i^T\Delta P+\Delta PA_i) z\}\ge 0$. Together with (\ref{eq:minDV}) we have $\min_{i\in\Q} DV_N^h(z;A_iz)\le -\kappa\|z\|,\forall z\in\R^n$, which completes the proof.
\end{proof}

\begin{thm}[Converse pm-PQCLF Theorem]
\label{thm:converseclf}
If system~(\ref{eq:ol}) is switching stabilizable, then it admits a pm-PQCLF.
\end{thm}
\begin{proof}
If system~(\ref{eq:ol}) is switching stabilizable, then it is exponentially discrete switching stabilizable (Theorem~\ref{thm:sstodss}). Furthermore, the $V^h_N$ with sufficiently small $h$ and sufficiently large $N$ is a pm-PQCLF for system~(\ref{eq:ol}) (Lemma~\ref{lem:dssclf}).
\end{proof}
Theorem~\ref{thm:converseclf} provides a formal justification for many existing works that have adopted quadratic or piecewise quadratic CLFs for simplicity or heuristic reasons~\cite{liberzon2003switching,WPD94,WPD98,HB98,SESP99,DBPL00,P03,HML08,LA09}. It allows us to only focus on pm-PQCLFs in the study of switching stabilizability for continuous-time SLSs. 

\section{Equivalent Characterizations for Switching Stabilizability}
The goal of this section is to prove the equivalence of the four switching stabilizability definitions and show that they all imply the existence of a pm-PQCLF. We first introduce several lemmas to show some key pairwise relations among them.

The relation between the most general switching stabilizability and feedback stabilizability in Filippov sense has not been adequately studied in the literature. It turns out that feedback stabilizability in Filippov sense implies switching stabilizability. Such a result justifies the generality of the switching stabilizability defined on measurable signals. The main idea of the proof is to think of Filippov solution as a solution to the relaxed system $(\mathcal{R})$ and use Lemma~\ref{lem:chattering} to construct measurable switching signals.
\begin{lem}
\label{lem:fstoss}
If system~(\ref{eq:ol}) is feedback stabilizable in Filippov sense, then it is switching stabilizable. 
\end{lem}
\begin{proof}
Assume system~(\ref{eq:ol}) is feedback stabilizable in Filippov sense. There exist a stabilizing feedback law $\nu:\R^n\to\Q$ and a constant $T>0$ such that $\|x(t;z,\nu)\|\le \frac{1}{2}\|z\|,\forall t\ge T,\forall z\in\R^n$. We now fix the finite time horizon $T$ and construct a stabilizing switching signal $\sigma:\R_+\to\Q$ recursively on intervals of length $T$. Let $\phi(\cdot)\triangleq x(\cdot;z,\sigma):\R_+\to\R^n$ be the state trajectory of system~(\ref{eq:ol}) under $\sigma$. Since the velocity of a Filippov solution can always be written as the convex combination of subsystem vector fields, i.e., $\dot{x}(t;z,\nu)=\sum_{i\in\Q} \alpha_i(t)A_ix(t;z,\nu)$ where $\sum_{i\in\Q}\alpha_i(t)=1, \forall t\in\R_+$, we can think of $x(\cdot;z,\nu):\R_+\to\R^n$ as a stabilizing trajectory of the relaxed system~$(\mathcal{R})$. By Lemma~\ref{lem:chattering}, $\forall z\in\R^n,\epsilon>0$, $\exists \sigma(z,\epsilon,\nu)\in\S_p^+$ s.t. $\|x(t;z,\sigma(z,\epsilon,\nu))-x(t;z,\nu)\|\le \epsilon\|z\|,\forall t\in[0,T]$, where the parenthesis in $\sigma(z,\epsilon,\nu)$ is used to emphasize the dependency of $\sigma$ on $z,\epsilon,\nu$. Let $\sigma_k\triangleq \sigma|_{[kT,(k+1)T]}:[0,T]\to\Q$ be the restriction of $\sigma$ on $[kT,(k+1)T]$. Consider the Filippov solution (also a relaxed trajectory) starting from the end point of the trajectory under $\sigma$ on the last interval, i.e. $x(\cdot;\phi(kT),\nu)$. By assumption, $\|x(T;\phi(kT),\nu)\|\le\frac{1}{2}\|\phi(kT)\|$. By Lemma~\ref{lem:chattering}, $\exists\sigma_k:[0,T]\to\Q$ where $\sigma_k\in\S_p^+$ s.t. $\|x(t;\phi(kT),\sigma_k)-x(t;\phi(kT),\nu)\|\le \frac{1}{2^{k+1}}\|\phi(kT)\|,\forall t\in[0,T]$. Thus, $\|\phi((k+1)T)\|=\|x(T;\phi(kT),\sigma_k)\|\le \|x(T;\phi(kT),\nu)\|+\|x(T;\phi(kT),\sigma_k)-x(t;\phi(kT),\nu)\|\le (\frac{1}{2}+\frac{1}{2^{k+1}})\|\phi(kT)\|,\forall k\in\N$ where $\phi(0)=z$. By construction, the state trajectory under $\sigma$ satisfies $\|\phi(t)\|\le \Pi_{k=0}^{\lfloor t/T\rfloor-1}(\frac{1}{2}+\frac{1}{2^{k+1}})\|z\|\le (\frac{3}{4})^{\lfloor t/T\rfloor-1}\|z\|\to 0$ as $t\to\infty$. Since $\sigma_k\in\S_p^+,\forall k\in\N$, we verified that $\sigma\in\S_p^+$.
\end{proof}

One sufficient condition for feedback stabilizability in Filippov sense is the existence of a pm-PQCLF~\cite{HML08}, in which a construction of stabilizing feedback law is also provided.
\begin{lem}[\cite{HML08}]
If system~(\ref{eq:ol}) admits a pm-PQCLF, then it is feedback stabilizable in Filippov sense.
\label{lem:clffsf}
\end{lem}

Based on the converse pm-PQCLF theorem and the two lemmas introduced above, we can claim the equivalence of switching stabilizability, feedback stabilizability in Filippov sense and the existence of a pm-PQCLF. It remains to establish their relation to exponential feedback stabilizability in S-H sense with bounded sampling rate. To this end, we use the admitted pm-PQCLF to generate a switching law and show that the switching law guarantees exponential stability of all the sampled closed-loop trajectories with sufficiently small intersampling time. The proof uses the decreasing condition~(\ref{cond:inf}) to show the exponential convergence of pm-PQCLF along the closed-loop $\pi$-trajectory under the constructed switching law.
\begin{lem} 
\label{lem:clffssh}
If system~(\ref{eq:ol}) admits a pm-PQCLF, then it is exponentially feedback stabilizable in sample-and-hold sense with bounded sampling rate.
\end{lem}
\begin{proof}
Let $V$ be a pm-PQCLF. By the property of pm-PQCLF, there exists $0<C_V^-<C_V^+<\infty$ such that $C_V^-\|z\|^2<V(z)<C_V^+\|z\|^2,\forall z\in\R^n$. According to the decreasing condition~(\ref{cond:inf}), there exists $\kappa>0$ such that 
\begin{align*}
\min_{i\in\Q} DV(z;A_iz) \le -3\kappa V(z),\forall z\in\R^n.
\end{align*}
For each $\kappa>0$, we can find an $h_0$ such that $0<h_0\le \kappa C_V^-/\max_{i\in\Q,k\in\N_m}\|A_i^T(A_i^TP_k+P_kA_i)+(A_i^TP_k+P_kA_i)A_i\|$ and $1-2\kappa h_0 \le e^{-2\kappa h_0}$. Let the switching law $\nu:\R^n\to\Q$ be 
\begin{align*}
\nu(z)=\argmin_{i\in\Q} DV(z;A_iz),\forall z\in\R^n. 
\end{align*}
Consider a sampling schedule $\pi=\{t_k\}_{k\in\N}$ with $d(\pi)<h_0$. It follows from the definition of S-H solution that for any $\tau\in(0,h_0)$,
\begin{align*}
V(x_\pi(\tau;z,\nu))=V(z)+\int_0^\tau DV(e^{A_{\nu(z)}t}z;A_{\nu(z)}e^{A_{\nu(z)}t}z)dt\\
=V(z)+\tau DV(e^{A_{\nu(z)}t}z;A_{\nu(z)}e^{A_{\nu(z)}t}z) \text{ for some } t\in(0,\tau),
\end{align*}
where the last equality is due to Mean Value Theorem. For $0<t<\tau<h_0$, the directional derivative in the last equation can be bounded as follows.
\begin{multline*}
DV(e^{A_{\nu(z)}t}z;A_{\nu(z)}e^{A_{\nu(z)}t}z)\le DV(z;A_{\nu(z)}z)+\\t\cdot V(z)/C_V^-\cdot\|A_{\nu(z)}^T(A_{\nu(z)}^TP_k+P_kA_{\nu(z)})+\\(A_{\nu(z)}^TP_k+P_kA_{\nu(z)})A_{\nu(z)}\|\le -2\kappa V(z).
\end{multline*}
Thus, the value of $V$ along closed-loop $\pi$-solution satisfies $V(x_\pi(\tau;z,\nu))\le (1-2\kappa\tau) V(z)\le e^{-2\kappa\tau}V(z),\forall z\in\R^n,\forall\tau\in (0,h_0)$. By iteratively applying the above inequality on intervals $[t_k,t_{k+1}],k\in\N$ of length less than $h_0$, we have $V(x_\pi(t;z,\nu))\le e^{-2\kappa t}V(z),\forall t\in\R_+,\forall z\in\R^n$. By the bound of $V$, $\|x_\pi(t;z,\nu)\|\le Ce^{-\kappa t}\|z\|, \forall t\in\R_+,\forall z\in\R^n$, where $C\triangleq (C_V^+/C_V^-)^{\frac{1}{2}}$.
\end{proof}

We are now ready to state the main result of this paper, namely, the equivalence among all the four switching stabilizability definitions and the existence of a pm-PQCLF. The proof of the main result is illustrated by the diagram in Fig.~\ref{fig:diagram}. 
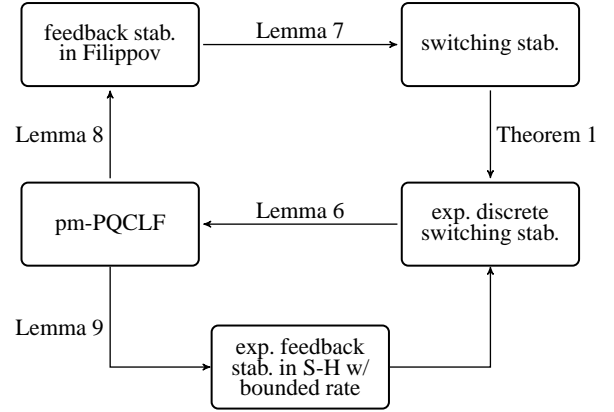
\begin{figure}[ht!]
\centering
\resizebox{0.9\linewidth}{!}{
\begin{tikzpicture}
 \node[block] (ctclf) {\Large pm-PQCLF};
 \node[right=2cm of ctclf] (temp) {};
 \node[block, right=2cm of temp] (dss) {\Large exp. discrete switching stab.};
 \node[block, above=2cm of ctclf] (fsf) {\Large feedback stab. in Filippov};
 \node[block, above=2cm of dss] (ss) {\Large switching stab.};
 \node[block, below=2cm of temp] (fssh) {\Large exp. feedback stab. in S-H w/ bounded rate};
 \path[link] (fsf) edge node [above] {\Large Lemma~\ref{lem:fstoss}} (ss); 
 \path[link] (ss) edge node [right] {\Large Theorem~\ref{thm:sstodss}} (dss); 
 \path[link] (dss) edge node [above] {\Large Lemma~\ref{lem:dssclf}} (ctclf); 
 \path[link] (ctclf) edge node [left] {\Large Lemma~\ref{lem:clffsf}} (fsf); 
 \draw[->,thick] (ctclf) |- node [pos=0.3, left] {\Large Lemma~\ref{lem:clffssh}} (fssh); 
 \draw[->,thick] (fssh) -| (dss); 
\end{tikzpicture}}
\caption{Relations of the statements in Theorem~\ref{thm:main}. ``stab.'' stands for stabilizability; ``exp.'' stands for exponentially.}
\label{fig:diagram}
\end{figure}
\begin{thm} 
\label{thm:main}
The following statements are equivalent for continuous-time switched linear system~(\ref{eq:ol}):
\begin{enumerate}[label=\roman*)]
\item It is switching stabilizable;
\item It is feedback stabilizable in Filippov sense;
\item It is exponentially feedback stabilizable in sample-and-hold sense with bounded sampling rate;
\item It is exponentially discrete switching stabilizable;
\item It admits a pm-PQCLF.
\end{enumerate}
\end{thm}
\begin{proof}
ii) $\Rightarrow$ i): It follows from Lemma~\ref{lem:fstoss} that if system~(\ref{eq:ol}) is feedback stabilizable in Filippov sense, then it is switching stabilizable. i) $\Rightarrow$ iv): It is shown in Theorem~\ref{thm:sstodss} that if system~(\ref{eq:ol}) switching stabilizable, then it is exponentially discrete switching stabilizable. iv) $\Rightarrow$ v): It is shown in Lemma~\ref{lem:dssclf} that if system~(\ref{eq:ol}) is exponentially discrete switching stabilizable, then it admits a pm-PQCLF. v) $\Rightarrow$ ii): It is proved in~ \cite{HML08} (Lemma~\ref{lem:clffsf}) that if system~(\ref{eq:ol}) admits a pm-PQCLF, then it is feedback stabilizable in Filippov sense. v) $\Rightarrow$ iii): It gives by Lemma~\ref{lem:clffssh} that if system~(\ref{eq:ol}) admits a pm-PQCLF, then it is exponentially feedback stabilizable in S-H sense with bounded sampling rate. iii) $\Rightarrow$ iv): It trivially holds by choosing sampling schedule $\pi$ with intersampling time of uniform length.
\end{proof}

Theorem~\ref{thm:main} shows the equivalence of the four switching stabilizability definitions and provides a unified sufficient and necessary condition, namely the existence of a pm-PQCLF, for all of them. We now have a guarantee that it suffices to only consider pm-PQCLFs in the stabilization of SLSs under various stabilizability notions.

\section{Conclusion}
This paper studies switching stabilization problems for continuous-time switched linear systems. We show the equivalence of the four switching stabilizability definitions and the existence of a pm-PQCLF. Such a result unifies the study of switching stabilizability under various assumptions on the switching control input. It also justifies many existing stabilization results that have used piecewise quadratic CLF for simplicity or heuristic reasons. Future work will focus on developing efficient algorithms to construct the proposed pm-PQCLF and the corresponding stabilizing feedback switching law.

\ifCLASSOPTIONcaptionsoff
  \newpage
\fi

\bibliographystyle{IEEEtran}
\bibliography{sls}
\end{document}